\documentclass[11pt,a4paper,onesided]{amsart}

\usepackage{amssymb}
\usepackage{amsmath}
\usepackage{pdfpages}

\usepackage{mathrsfs}
\usepackage{stmaryrd} 

\newcommand{\dd}{\mathrm{d}}
\newcommand{\I}{\mathrm{I}}
\newcommand{\II}{\mathrm{I\!I}}

\newcommand{\F}{\mathscr{F}}
\newcommand{\Ff}{\mathscr{F}_{\!f}}

\newcommand{\T}{\mathrm{T}}
\newcommand{\llangle}{\left\langle}
\newcommand{\rrangle}{\right\rangle}

\newcommand{\deltatnought}{\left.\frac{\partial}{\partial t}\right\vert_{t=0}}
\newcommand{\Order}{\mathscr{O}}
\newcommand{\grad}{\mathrm{grad}} 
\newcommand{\tangentpart}{\mathfrak{t}}

\newtheorem{theorem}{Theorem}

\theoremstyle{definition}

\newtheorem{remark}[theorem]{Remark}

\begin{document}

\title[A Characterisation of Manhart's Relative Normal Vector Fields.]{A Characterisation of Manhart's\\ Relative Normal Vector Fields.}

\author[S. Verpoort]{Steven Verpoort}

\dedicatory{(K.U.Leuven, Belgium, and  Masaryk University\,/\,Eduard \v{C}ech Center, Czech Republic.)}

\thanks{
The author, who was employed at K.U.Leuven during the commencement of this work, while he was supported by Masaryk University (Brno) during its conclusion, is thankful to both these institutions. 
This research was partially supported by 
the Research Foundation Flanders (project G.0432.07) and the Eduard \v{C}ech Center for Algebra and Geometry (Basic Research Center no. LC505).
}

\date{}

\address{K.U.Leuven, Departement Wiskunde, Celestijnenlaan 200B, 3001 Heverlee, Belgium.}

\curraddr{Masaryk University, Department of Mathematics and Statistics, 
Kotl\'a\v{r}sk\'a 2, 611 37 Brno, Czech Republic.}

\maketitle

\section{Introduction.}

In this article a relation between curvature functionals for surfaces in the Euclidean space and area functionals in relative differential geometry will be given.  

Relative differential geometry can be described as the geometry of surfaces in the affine space, endowed with a distinguished ``relative normal vector field'' which generalises the notion of unit normal vector field $N$ from Euclidean differential geometry. A concise review of relative differential geometry will be presented in \S\,2 (we refer to \cite{leichtweiss1993,simon_schwenk_viesel_1991} for more details). 

The main result, to which the title of this article refers, will be given in the third section. Here we consider, for a function $f$ of two variables, relative normal vector fields of the form
\begin{flalign*}
(\dagger)&\hspace{5cm}
f(H,K)\,N-\grad_{\II}(f(H,K))
&&
\end{flalign*}
for non-degenerate surfaces in the Euclidean three-dimensional space. A comparison of the variation of the curvature functional  
\begin{flalign*}
(\ast)&\hspace{6cm}
\int f(H,K)\,\dd\Omega&&
\end{flalign*}
with the relative area functional obtained from the above relative normal vector field, results in a distinguishing property for the one-parameter family of relative normal vector fields which was introduced by F.\ Manhart, and which is obtained by choosing $f(H,K)=|K|^{\alpha}$ (where we will assume that $\alpha\neq 1$). More precisely, the following will be shown in theorem~\ref{thm:manhart}:

``\textit{The curvature functionals $(\ast)$ for which the critical points coincide with the relative-minimal surfaces with respect to the relative normal vector field $(\dagger)$, are essentially those obtained from Manhart's family}.\!''

In the fourth section, we give a characterisation of the sphere by means of relations between the support function and the curvatures.

In the last section, we combine the previously described results and arrive at a variational characterisation of the sphere.

\section{Preliminaries on Relative Differential Geometry.}

The starting point of this brief notice on relative differential geometry is the observation that several important differential-geometric invariants for  surfaces in the Euclidean three-dimensional space are already determined by the affine structure of the ambient space along with knowledge of the unit normal vector field $N$. For instance, the shape operator (or Weingarten operator) can already be constructed from this information, and hence the mean and Gaussian curvature and the principal directions can be found as well. 

We will denote $\mathrm{D}$ for the standard connection of the affine space $\mathbb{A}^3$ and $\mathrm{det}(\cdot,\cdot,\cdot)$ for a fixed volume form.
The set of all tangent resp.\ $\mathbb{A}^3$--valued vector fields on a surface $M$ in  $\mathbb{A}^3$ will be denoted by $\mathfrak{X}(M)$ resp.\ $\overline{\mathfrak{X}}(M)$, and $\mathfrak{F}(M)$ stands for the collection of all real-valued functions on $M$.

We will study surfaces in the affine space which are \textit{non-degenerate} 
and \textit{oriented}. These requirements precisely mean that the surfaces have a nowhere degenerate Blaschke metric and that for every point $p$ of the surface, one of the two half-spaces bounded by the plane which is tangent to the surface at $p$ is designated as the positive one. For ovaloids, we will always consider the half-space which contains the surface as the positive one.

A vector field $y\in\overline{\mathfrak{X}}(M)$ along a non-degenerate oriented surface $M$ will be called a \textit{relative normal vector field} if for all $p\in M$ there holds $y_p \notin \T_p M$ and for all $V\in\mathfrak{X}(M)$, the vector field $\mathrm{D}_V y$ is again tangent to the surface. As such we can define the \textit{relative shape operator}
\[
A_{(y)} : \mathfrak{X}(M) \rightarrow \mathfrak{X}(M) : V \mapsto -\mathrm{D}_V y\,,
\]
and the \textit{relative mean curvature}  $H_{(y)}=\frac{1}{2}\mathrm{trace}(A_{(y)})$. The \textit{relative area element} is defined to be $\dd\Omega_{(y)}=y\lrcorner\,\mathrm{det}$. Further, the 
\textit{relative area} of $M$ w.r.t. $y$ is defined as 
\begin{equation}
\label{eq:areayM}
\textrm{Area}_{(y)}(M) = \int_M \dd\Omega_{(y)}\,.
\end{equation}

It should be remarked that besides the Euclidean normal vector field, also the Blaschke normal (or equi-affine normal vector field) and the position vector field (or centro-affine normal vector field) can be interpreted in this framework.

An ovaloid $\mathscr{M}$ which has been endowed with a relative normal vector field $y$ determines a relative normal vector field $y^{M}$ on every other non-degenerate oriented surface $M$. Herefore, we recall that the 
\textit{Peterson mapping} $\mathscr{P}:  M \rightarrow \mathscr{M}$ sends a point $m$ of the surface $M$ to the unique point $\mathscr{P}(m)$ of the ovaloid $\mathscr{M}$ such that $\mathrm{T}_m {M}$  is parallel to $\mathrm{T}_{\mathscr{P}(m)}\mathscr{M}$ whereby the positive half-spaces correspond up to translation. Now we define for every ${m}\in {M}$,
\begin{equation}
\label{eq:ytilde}
\left(y^{{M}}\right)({m}) 
\,=\, y(\mathscr{P}({m}))\quad\textrm{up to translation.}
\end{equation}
(See also figure \ref{fig:vecfield} on page \pageref{fig:vecfield}.)

Consequently, an ovaloid $\mathscr{M}$ endowed with a relative normal vector field $y$ determines a \textit{relative area functional} 
\begin{equation}
\label{eq:areaytildeM}
\mathrm{Area}_{(y)} : \left\{\textrm{\,non-degenerate oriented surfaces\rule{0pt}{14pt}\,}\right\} \rightarrow \mathbb{R}
: M \mapsto 
\mathrm{Area}_{(y)} ({M}) := 
\mathrm{Area}_{(y^{M})} (M)
\end{equation}
(where $y^{M}$ is defined in (\ref{eq:ytilde}) and 
the relative area of $M$ w.r.t.\ $y^M$ is determined as in (\ref{eq:areayM})), and the corresponding variational problem will be studied.

By a \textit{deformation} $\mu$ of a non-degenerate oriented surface $M$
will be understood a smooth mapping (for some
$\varepsilon > 0$)
\[
\mu: \,\left]\,-\varepsilon\, ,\, \varepsilon \,\right[ \,\times M \rightarrow
\mathbb{A}^{3} : (t,m) \mapsto \mu_t(m)\,,
\]
such that for all $m\in M$ there holds $\mu_0(m)=m$, and 
for some compact set $M_0\subseteq M$, 
for all $m\in M\setminus M_0$ and all $t\in\left]\,-\varepsilon\, ,\, \varepsilon \,\right[$,
there holds $\mu_t(m)=m$.
It is clear that the \textit{deformed surface}
$\mu_t(M)$ is also non-degenerate for small values of $\left|\varepsilon\right|$ and that the orientation on $M$ determines an orientation on 
$\mu_t(M)$. The $\mathbb{A}^{3}$-valued vector field
sending $m\in M$ to $Z(m)=\deltatnought\!\mu_t(m)$
will be called the \textit{deformation vector field}.
Furthermore, the \textit{variation of the relative area}
is defined by
\[
\delta \mathrm{Area}_{(y)} =\deltatnought \mathrm{Area}_{(y)}(\mu_t(M)).
\]
The following theorem is well-known. 

\begin{figure}[!t]
\begin{center}
\framebox{\includegraphics[angle=0,width=11cm]{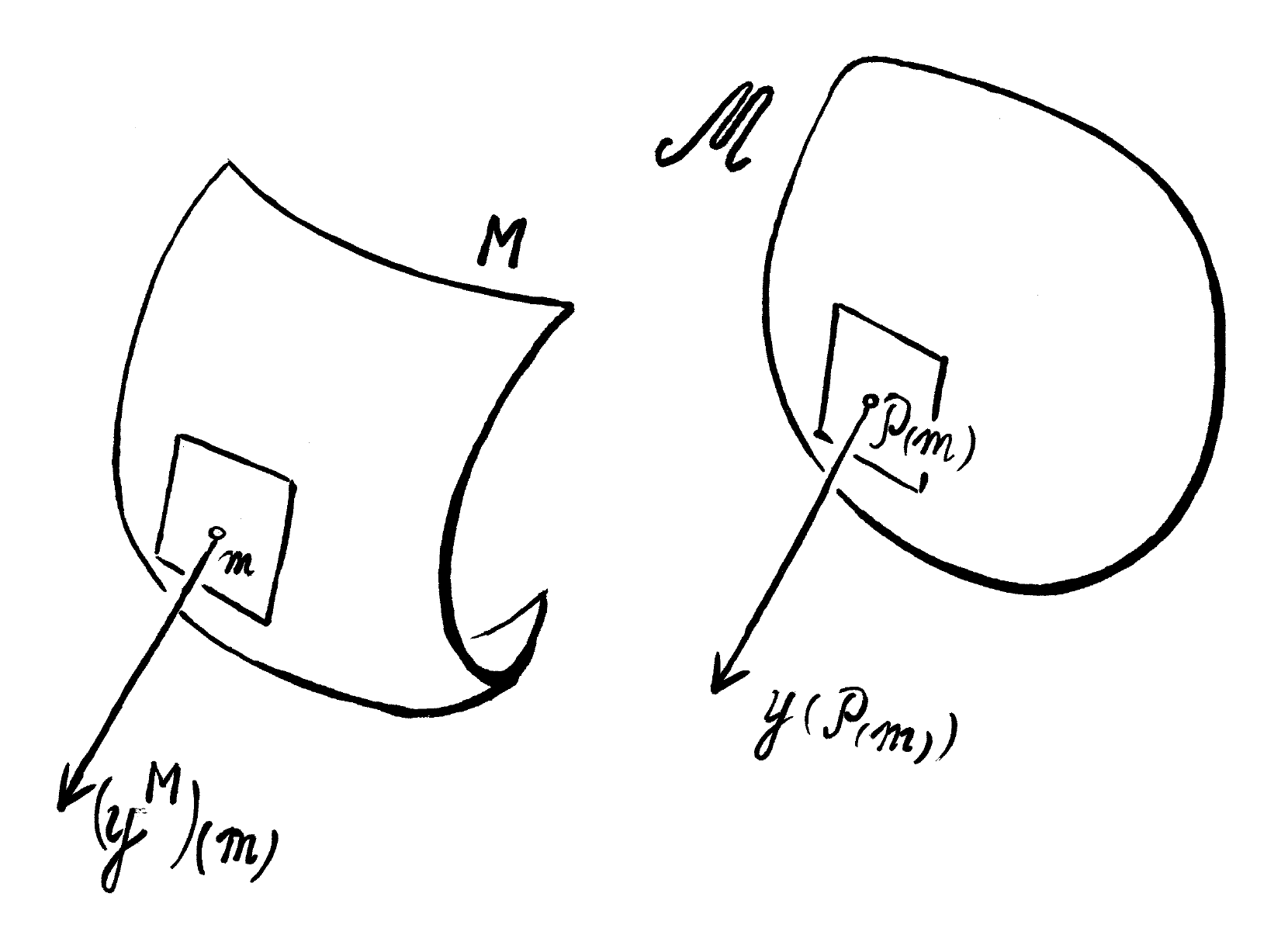}}
\caption{Construction of the vector field $y^{M}$ on 
$M$ starting from the vector field $y$ on $\mathscr{M}$.
}
\label{fig:vecfield}
\end{center}
\end{figure}

\begin{theorem}
Let $\mathscr{M}$ be an ovaloid with a relative normal vector field $y$. 
Under a deformation $\mu$ of a non-degenerate oriented surface $M$ with deformation vector field $Z=\varphi\,y^{M}+Z^{\tangentpart}$ (where $Z^{\tangentpart}$ is tangent to $M$), the variation of the relative area is given by
\begin{equation}
\label{eq:delta_rel_ar}
\delta \mathrm{Area}_{(y)} = -2\int_{M} \,\varphi\,H_{(y^{M})}\ \dd\Omega_{(y^{M})}\,.
\end{equation}
\end{theorem}

\begin{remark}
The most important case occurs when
the vector field $y$ is simply minus the position vector field on the
 \textit{gauge ovaloid} (``Eichfl\"ache'') $\mathscr{M}$ with respect to a designated origin in the interior of $\mathscr{M}$. In this case the above construction is merely a small modification of the construction of several well-known objects in classical differential geometry. For instance, the Peterson mapping $\mathscr{P}: M \rightarrow\mathscr{M}$ between the ovaloid and the gauge ovaloid takes the place of the Gauss map, and up to identification of parallel tangent planes there holds $\dd\mathscr{P}=-A_{(y)}$. This construction of a relative normal vector field starting from a gauge ovaloid follows the original development of relative differential geometry in the spirit of Minkowski's article \cite{minkowski1903}, in which the relative area of a non-degenerate oriented surface $M$ with respect to the gauge ovaloid $\mathscr{M}$ is denoted as $3\,V(M,M,\mathscr{M})$ and called a \textit{mixed volume} of $M$ and $\mathscr{M}$.
\end{remark}

\begin{remark}
In the last-mentioned situation, let us introduce a Euclidean structure  which is compatible with the determinant function, and denote $F$ for the Euclidean support function of $\mathscr{M}$, which is seen as a function on the sphere by identification of the sphere and $\mathscr{M}$ under Peterson correspondence. Now, for any non-degenerate oriented surface $M$ 
(with Gauss mapping $\mathscr{G}$),  there holds
\[
\textrm{relative area of $M$ w.r.t.\ $\mathscr{M}$}
\ =\
\int_M \left( F\circ\mathscr{G}\right)\,\dd\Omega\,,
\]
where $\dd\Omega$ stands for the classical area element of the surface $M$ (which, \textit{e.g.}, can be constructed from its first fundamental form).
Thus we see that the area functional from relative differential geometry w.r.t.\ the gauge ovaloid $\mathscr{M}$ coincides with the anisotropic area 
functional determined by the function $F$, which has recently attracted a 
renewed interest (see, \textit{e.g.}, \cite{palmer2008}; \cite{wulff1901} \S\,8).
\end{remark}

\begin{remark}
Relative-minimal surfaces were first introduced by E.\ M\"uller in \cite{muller1921} (some helpful comments on this article are contained in
\cite{simon_schwenk_viesel_1991}, \S\,4.6.1), whereas they have been characterised variationally only afterwards by A. Duschek \cite{duschek1926} (who also refers to \cite{blaschkeII} p.\ 205, and a lecture by J. Radon). 

It can be seen that if a surface is relative-minimal w.r.t.\ a second surface, then the converse holds as well. Moreover, such pairs of surfaces were already studied in the context of infinitesimal isometric deformations by L.\ Bianchi, who called them \textit{associate surfaces}.
\end{remark}

\begin{remark}
\label{rmk:M_ovaloid}
Assume that $\mathscr{M}$ is a non-degenerate, oriented surface but not necessarily an ovaloid. Then the above results remain valid under the explicit requirement that for each point $m$ of the surface $M$ there uniquely exists a point $\mathscr{P}(m)$ of $\mathscr{M}$ such that $\mathrm{T}_m {M}$ is parallel to $\mathrm{T}_{\mathscr{P}(m)}\mathscr{M}$ whereby the positive half-spaces correspond up to translation. 
\end{remark}

\section{Manhart's Family of Relative Normal Vector Fields.}

We will assume in the sequel that $\mathscr{D}$ is an open set in $\mathbb{R}^2$ which contains the set
$\mathscr{A} = 
\left\{\,\rule{0pt}{10pt}(u,v)\in\mathbb{R}^2\,\vert\,u^2\geqslant v\ \textrm{and}\ v\neq 0 \,\right\}$.
Let now $f:\mathscr{D}\rightarrow\mathbb{R}:(u,v)\mapsto f(u,v)$ be an arbitrary smooth, nowhere vanishing function of two variables. For any non-degenerate surface $M$ in the Euclidean space, we define a vector field along $M$ by
\begin{equation}
\label{eq:relnorm}
N_f^M = f(H,K)\,N-\grad_{\II}(f(H,K)).
\end{equation}
Here $N$, $H$ and $K$ stand, respectively, for the Euclidean normal vector field, the mean curvature (which equals $H_{(N)}$) and the Gaussian curvature of the surface. The index $\II$ signals that the gradient has been determined w.r.t.\ the (Euclidean) second fundamental form geometry. It is precisely by choosing $-\grad_{\II}(f(H,K))$ as the tangential part of the vector field $N_f^M$ with prescribed normal part $f(H,K)\,N$, that a relative normal vector field results.

In the sequel, we will use some tensors which have been defined w.r.t.\ the second fundamental form geometry. For instance, the \textit{difference tensor} between the Levi-Civita connections of the second and the first fundamental form is defined as
\[
L : \mathfrak{X}(M) \times \mathfrak{X}(M) \rightarrow \mathfrak{X}(M)
 : (V,W) \mapsto \nabla^{\II}_V W - \nabla_V W.
\]
Some simple calculations related to the second fundamental form geometry which will be used below, can be found in \S\,2.3.2 of \cite{verpoort2008}.

Furthermore, if $\varphi$ is a function on $M$, then the \textit{Hessian operator} of $\varphi$ is defined as $\textrm{Hs}_{\varphi}: \mathfrak{X}(M)\mapsto\mathfrak{X}(M): V \mapsto \nabla_V (\grad \varphi)$. 
The \textit{Laplacian} of $\varphi$ (notation $\Delta \varphi$) is defined as the trace of this operator.
Of course, we can also define $\mathrm{Hs}_{\varphi}^{\II}$ resp.\ $\Delta_{\II}\varphi$ w.r.t.\ the second fundamental form geometry.

The relative shape operator determined by the relative normal vector field (\ref{eq:relnorm}) satisfies, for every vector field $V\in\mathfrak{X}(M)$,
\begin{eqnarray*}
A_{(N_f^M)}(V)
&=&
-\mathrm{D}_V \Big(f(H,K)N-\grad_{\II}(f(H,K))\Big)
\\
&=&
f(H,K)\,A(V) + \textrm{Hs}_{(f(H,K))}^{\II}(V)-L\big(\,\grad_{\II}(f(H,K))\,,\,V\,\big)\,.
\end{eqnarray*}
The relative mean curvature can be calculated as 
\begin{equation}
\label{eq:Hrel}
H_{(N_f^M)}
= 
f(H,K)\,H + \frac{1}{2}\,\Delta_{\II}( f(H,K) ) - 
\frac{1}{4}\,\II\big( \grad_{\II}(\log |K|) \,,\, \grad_{\II} (f(H,K))\, \big)\,.
\end{equation}
The relative normal vector field $N_f^M$ on $M$ induces a corresponding relative area functional
\begin{equation*}
\mathrm{Area}_{(N_f^M)} : \left\{\textrm{\,non-degenerate oriented surfaces\rule{0pt}{14pt}\,}\right\} \rightarrow \mathbb{R}
: \widetilde{M} \mapsto 
\mathrm{Area}_{(N_f^M)} (\widetilde{M}) \,.
\end{equation*}

Another functional of interest (which does not depend on $M$) is the curvature functional
\begin{equation*}
\Ff :\left\{\textrm{\,non-degenerate oriented surfaces\rule{0pt}{14pt}\,}\right\} \rightarrow \mathbb{R}
: \widetilde{M} \mapsto \F_f(\widetilde{M}) = \int_{\widetilde{M}} f(\widetilde{H},\widetilde{K}) \,\dd\widetilde{\Omega}\,.
\end{equation*}
(Here $\widetilde{H}$ and $\widetilde{K}$ stand for the mean and Gaussian curvature of $\widetilde{M}$, and $\dd\widetilde{\Omega}$ stands for the classical area element of $\widetilde{M}$.)

Of course there holds $\Ff(M)=\mathrm{Area}_{(N_f^M)}(M)$. But for a special choice of the function $f$, the quantities
\[
\delta \Ff
\qquad\textrm{and}\qquad
\delta \mathrm{Area}_{(N_f^M)} 
\]
are also connected by a simple relation for any deformation $\mu$ of $M$. This will be described in theorem \ref{thm:manhart} below.
The distinction between the last two quantities essentially originates from different approaches towards the question as how a relative normal vector field should be changed under a deformation of the surface. (Cf.\ also
\cite{wiehe1998} \S\,3.1 for some remarks on this issue.)

We can equip every deformed surface $\mu_t(M)$ with \textit{its own} relative normal vector field defined by (\ref{eq:relnorm}) (in which $N$, $K$ and $H$ have to be taken the Euclidean normal, the Gaussian curvature and the mean curvature of the \textit{deformed} surface  $\mu_t(M)$). In this way, the quantity $\delta \Ff$ can be considered as the first-order variation of the relative area along the path of deformed surfaces each of which has been endowed with its own relative normal vector field. 

(The more usual interpretation of $\delta \Ff$ of course does not use relative-differential-geometric concepts, but is rather the variation of the curvature functional $\int f(H,K)\,\dd\Omega$ under the deformation of the surface.)

On the other hand, the deformed surfaces can also be equipped with the \textit{initial} relative normal vector field, which has been transferred to the deformed surface by translation from a point on the initial surface to the point on the deformed surface which agrees under the Peterson correspondence (as explained above). The first-order change of the relative area which has been calculated in this way, is precisely 
$\delta \mathrm{Area}_{(N_f^M)}$.

The collection of all relative normal vector fields (\ref{eq:relnorm}) which is obtained by choosing the function $f$ to be $f(u,v)=|v|^{\alpha}$ for some $\alpha\neq 1$, wil be called \textit{Manhart's one-parameter family of relative normal vector fields}.
These vector fields have been introduced by F. Manhart in \cite{manhart1983}, who observed that the Euler-Lagrange equation of the functional $\int\sqrt{K}\,\dd\Omega$, which was derived by E. Gl\"assner in \cite{glassner1974}, precisely means that the relative mean curvature w.r.t.\ (\ref{eq:relnorm}), for $f(u,v)= \sqrt{|v|}$, vanishes. 

The relative differential geometry of surfaces w.r.t.\ this one-parameter family of relative normal vector fields has been investigated in a series of articles by F. Manhart \cite{manhart1982}--\cite{manhart1989_bis}. Particularly, some classification results for special classes of surfaces which are relative sphere w.r.t. such relative normal vector fields are described in these articles.

The above-mentioned property is distinctive for Manhart's family of relative normal vector fields, as is explained in the following theorem.

\begin{theorem}
\label{thm:manhart}
(i). Assume that $f(u,v)=|v|^{\alpha}$ for some $\alpha\in\mathbb{R}\setminus\{1\}$, \textit{i.e.}, that the relative normal vector field
(\ref{eq:relnorm}) belongs to Manhart's family of relative normal vector fields. Then for any deformation of any non-degenerate oriented surface $M$ there holds 
\begin{equation}
\label{eq:cond_mah}
\delta\, \Ff = 0 \quad \Leftrightarrow \quad \delta\, \textrm{Area}_{(N_f^M)} = 0\,.
\end{equation}
(ii). Moreover, the above circumstance can \textit{only} occur if, up to a factor, the relative normal vector field
(\ref{eq:relnorm}) belongs to Manhart's family of relative normal fields. 

More precisely, if $f:\mathscr{D}\rightarrow\mathbb{R}:(u,v)\mapsto f(u,v)$ is an arbitrary smooth, nowhere vanishing function of two variables such that the equivalence (\ref{eq:cond_mah}) holds, then there exist constants $q_1,q_2,q_3\in\mathbb{R}\setminus\{0\}$ and 
$\alpha\in\mathbb{R}\setminus\{1\}$ such that 
\begin{equation}
\label{eq:f}
f(u,v)=
\left\{
\begin{array}{lp{2mm}l}
q_1\,|v|^{\alpha} && \textrm{(for $u>0$ and $v>0$)\,;}\\
q_2\,|v|^{\alpha} && \textrm{(for $u<0$ and $v>0$)\,;}\\
q_3\,|v|^{\alpha} && \textrm{(for $v<0$)\,,}
\end{array}
\right.
\end{equation} 
is satisfied.
\end{theorem}

\begin{proof}
We will first establish part (ii).\ of the theorem.
Let us begin by writing down the variational formulae for the functionals 
$\Ff$ and $\textrm{Area}_{(N_f^M)}$, for a deformation $\mu$ of $M$ with deformation vector field
\[
Z = \varphi N_f^M = \varphi\, f(H,K)\,N + Z^{\tangentpart}
\]
where $Z^{\tangentpart}$ is tangent to the surface $M$. 
We will write these variational formulae as
\[
\left\{
\begin{array}{rcl}
\delta \textrm{Area}_{(N_f^M)} 
&=& 
\displaystyle \int_M \varphi\, f(H,K)\,\Phi\,\dd\Omega\,;\\
\delta \Ff
&=& 
\displaystyle \int_M \varphi\, f(H,K)\,\Psi\,\dd\Omega\,,\rule{0pt}{22pt}\\
\end{array}
\right.
\]
where $\Phi$ and $\Psi$ are certain differential invariants of the surface which depend on the function $f$, which will be calculated below.

The condition, that for any deformation of any surface the equivalence (\ref{eq:cond_mah}) holds, precisely means that for every non-degenerate surface there can be found a non-zero constant $C$ such that 
\begin{equation}
\label{eq:phi_psi}
\Phi=C\,\Psi\,.
\end{equation}
It can easily be seen that this constant $C$ is independent of the surface: for if on two non-degenerate surfaces  $M_{\sharp}$ and $M_{\dagger}$ the relation $\Phi=C_{\sharp}\,\Psi$ resp.\ $\Phi=C_{\dagger}\,\Psi$ is satisfied, then the corresponding condition can hold on the surface $M_{\sharp} \cup M_{\dagger}$ only in case $C_{\sharp} =C_{\dagger}$.

From a combination of (\ref{eq:delta_rel_ar}) and (\ref{eq:Hrel}) we obtain
\begin{eqnarray}
\nonumber
\Phi &=& -2\, H_{(N_f^M)} \\
\nonumber
&=&
-2\,f(H,K)\,H - f_u(H,K)\,\Delta_{\II}H- f_v(H,K)\,\Delta_{\II}K\\
\nonumber
&&
-f_{uu}(H,K)\,\II(\grad_{\II}H,\grad_{\II}H)
+\left(\frac{f_u(H,K)}{2\,K} -2\,f_{uv}(H,K)\right)\,\II(\grad_{\II}H,\grad_{\II}K)
\\
\label{eq:Phi}
&&
+\left(\frac{f_v(H,K)}{2\,K}-f_{vv}(H,K)\right)\,\II(\grad_{\II}K,\grad_{\II}K)\,.
\end{eqnarray}
An expression for the quantity $\Psi$ can be obtained from the variational formula for the functional $\int f(H,K)\,\dd\Omega$, which can be found, \textit{a.o.}, in \cite{verpoort2008}, theorem 1.6. This expression can be rewritten by making us of \cite{verpoort2008}, lemma 2.6:
\begin{eqnarray}
\nonumber
\Psi &=&
f_u(H,K)\,(2\,H^2-K)
+2\,H\,K\,f_v(H,K)-2\,f(H,K)\,H
\\
\nonumber
&&
+\frac{1}{2} \Delta \big(f_u(H,K)\big) 
+ K\, \textrm{tr}_{\II} \textrm{Hess}_{(f_v(H,K))}
\\
\nonumber
&=&
f_u(H,K)\,(2\,H^2-K)
+2\,H\,K\,f_v(H,K)-2\,f(H,K)\,H
\\
\nonumber
&&
+\frac{1}{2}\,f_{uuu}(H,K)\,\llangle\grad H,\grad H\rrangle
+f_{uuv}(H,K)\,\llangle\grad H,\grad K\rrangle
\\
\nonumber
&&
+\frac{1}{2}\,f_{uvv}(H,K)\,\llangle\grad K,\grad K\rrangle
+ K\,f_{uuv}(H,K)\,\II(\grad_{\II}H,\grad_{\II}H)
\\
\nonumber
&&
+\left( 2\,K\,f_{uvv}(H,K)+\frac{1}{2}f_{uv}(H,K)\right)\,\II(\grad_{\II}H,\grad_{\II}K)
\\
\nonumber
&&
+\left( K\,f_{vvv}(H,K)+\frac{1}{2}f_{vv}(H,K) \right)\,\II(\grad_{\II}K,\grad_{\II}K)
\\
\label{eq:Psi}
&&
+\frac{1}{2} f_{uu}(H,K)\,\Delta H + \frac{1}{2} f_{uv}(H,K)\,\Delta K
+K\,f_{uv}(H,K)\,\Delta_{\II}H +K\,f_{vv}(H,K)\,\Delta_{\II}K\,.
\end{eqnarray}
Before the reader who might attempt to solve the differential equation $\Phi = C\,\Psi$ runs away with fright and horror, we should stress that this differential equation has to be approached from an unusual perspective: instead of trying to find all surfaces which satisfy the equation for a given $f$, our goal is \textit{to find all functions} $f$ for which the differential equation is \textit{automatically} satisfied. Particularly, it will already be sufficient to evaluate the equation in a small number of simple surfaces to establish the theorem.

\textit{Case 1.} Let us first investigate this equation for surfaces for which $H$ and $K$ are strictly positive.  

If we evaluate the equation (\ref{eq:phi_psi}) on a sphere of radius $\frac{1}{x}$, we find that necessarily, for all $x >0$,
\begin{equation}
\label{eq:sphere}
f(x,x^2)= q_1\,x^{2(C-1)/C}\,,
\end{equation}
where $q_1$ is a constant.

Next, we will consider for five real numbers $a,b,c,\ell_1$ and $\ell_2$ (the last two of which are strictly positive), a one-parameter family of surfaces which is obtained by deforming a neighbourhood of the origin of the paraboloid
\[
M\ \ \longleftrightarrow\ \  z=\frac{1}{2}\,\ell_1\,x^2 + \frac{1}{2}\,\ell_2\,y^2, 
\]
which will be immersed by 
\[
\xi_0 : \mathbb{R}^2 \rightarrow\mathbb{E}^3: (x,y) \mapsto(x,y,\frac{1}{2}\,\ell_1\,x^2 + \frac{1}{2}\,\ell_2\,y^2)\,.
\]
We let $N$ stand for the upward unit normal vector field of $M$. The family 
$ \{\,M_t\,\}$ of surfaces which will be considered is, for some neighbourhood $\mathscr{B}$ of the origin, immersed by
\[
\xi_{t}: \mathscr{B}\subseteq\mathbb{R}^2 \rightarrow M_{t}\subseteq\mathbb{E}^3: (x,y) \mapsto
\xi_0(x,y)+t\left(a\,x^4+b\,x^2\,y^2+c\,y^4\right) N_{(x,y)}\,.
\]
This one-parameter family $\xi$ of immersions can be seen as a deformation of the initial immersion $\xi_0$. For every $t$, the relation $\Phi=C\,\Psi$ is satisfied by the surface $M_t$. We will be able to finish the proof for case 1 by exploiting the fact that the first-order variation of the quantities $\Phi$ and $C\,\Psi$ under this deformation, evaluated in the point $(0,0)$, are equal.
For several tensorial invariants, it will already be sufficient to calculate them for the \textit{initial} surface $M$.

The first fundamental form of $M$ is
\[
\I =
\left\lgroup
\begin{array}{ccc}
1+(\ell_1\,x)^2       &\rule{0pt}{14pt} &  \ell_1\,\ell_2\,x\,y \\
\ell_1\,\ell_2\,x\,y  &\rule{0pt}{14pt} &  1+(\ell_2\,y)^2
\end{array}
\right\rgroup
\]
and the second fundamental form is given by
\[
\II =
\frac{1}{\sqrt{1+(\ell_1\,x)^2+(\ell_2\,y)^2}}
\left\lgroup
\begin{array}{ccc}
\ell_1    &\rule{0pt}{14pt} &  0\\
0         &\rule{0pt}{14pt} &  \ell_2
\end{array}
\right\rgroup\,.
\]
The Gaussian and the mean curvature of the initial surface $M$ are given by
\[
\left\{
\begin{array}{rcl}
K &=&\displaystyle \frac{\ell_1\,\ell_2}{\left(1+(\ell_1\,x)^2+(\ell_2\,y)^2\right)^2}\\
H &=&\displaystyle \frac{\ell_1\,\left(1+(\ell_2\,y)^2\right)+
\ell_2\,\left(1+(\ell_1\,x)^2\right)}{2\,\left(1+(\ell_1\,x)^2+(\ell_2\,y)^2\right)^{\frac{3}{2}}}\rule{0pt}{26pt}
\end{array}
\right.
\]
All Christoffel symbols of the surface $M$, both w.r.t.\ the first and the second fundamental form, vanish at the origin.

We will denote by $\delta H$ resp. $\delta K$ the variation of the mean and the Gaussian curvature under the deformation $\xi$ of $M$. More generally the prefix $\delta$ will be used to indicate the variation of a tensor under this deformation $\xi$ of $M$ (as defined in \cite{verpoort2008}, pp. 1--5). 
These variations can be calculated from \cite{verpoort2008}, theorem 1.3:
\[
\left\{
\begin{array}{rcl}
\delta H & = & \frac{1}{2} \Delta \left(a\,x^4 + b\, x^2\,y^2 + c\,y^4\right)
+(2\,H^2-K)\,\left(a\,x^4 + b\, x^2\,y^2 + c\,y^4\right)\,;\\
\delta K & = & 
2\,K\,H\,\left(a\,x^4 + b\, x^2\,y^2 + c\,y^4\right)
+
K\,\Delta_{\II} \left(a\,x^4 + b\, x^2\,y^2 + c\,y^4\right)
\rule{0pt}{15pt}
\\
&&\qquad\qquad\qquad\qquad\qquad\qquad\qquad
+
\frac{1}{2}\,
\II\left( \grad_{\II}\left(a\,x^4 + b\, x^2\,y^2 + c\,y^4\right),\grad_{\II}\textrm{log}\,K \right)\,.
\rule{0pt}{15pt}
\end{array}
\right.
\]
These Laplacians are particularly easy to calculate up to $\Order(\|(x,y)\|^4)$ because of the vanishing of the Christoffel symbols at the origin. In this way there results
\begin{equation}
\label{eq:deltaHdeltaK}
\left\{
\begin{array}{rcl}
\quad\delta H &=& (6\,a+b)\,x^2+(b+6\,c)\,y^2 + \Order(\|(x,y)\|^4)\,;\\
\quad\delta K &=& 
 (12\,a\,\ell_2+2\,b\,\ell_1)\,x^2
+(2\,b\,\ell_2+12\,c\,\ell_1)\,y^2 + \Order(\|(x,y)\|^4)\,.\rule{0pt}{15pt}
\end{array}
\right.
\end{equation}

From \cite{verpoort2008}, theorem 4.1, follows that $\delta\II = \Order(\|(x,y)\|^2)$, and consequently $(\delta \II_{i\,j})_{(0,0)}=0$ and $(\partial_k(\delta \II_{i\,j}))_{(0,0)}=0$. This implies that $(\delta \Gamma_{i\,j}^{\II\,k})_{(0,0)}=0$ and hence there holds
$\delta (\Delta_{\II} \varphi)_{(0,0)}=(\Delta_{\II} (\delta\varphi))_{(0,0)}$ for every function $\varphi$. The similar result for $\Delta$ holds as well.

Furthermore, it can be seen that
\[
\left\{
\begin{array}{rcl}
(\delta H)_{(0,0)} &=& 0\rule{0pt}{14pt}\,;\\
\big(\delta(\grad H)\big)_{(0,0)} &=&0\,;\rule{0pt}{14pt}\\ 
\big(\delta(\grad_{\II} H)\big)_{(0,0)}&=&0\,;\rule{0pt}{14pt}\\
\end{array}
\right.
\qquad\textrm{and}\qquad
\left\{
\begin{array}{rcl}
(\delta K)_{(0,0)} &=& 0\rule{0pt}{14pt}\,;\\
\big(\delta(\grad K)\big)_{(0,0)} &=&0\,;\rule{0pt}{14pt}\\ 
\big(\delta(\grad_{\II} K)\big)_{(0,0)}&=&0\,.\rule{0pt}{14pt}\\
\end{array}
\right.
\]

Consequently, we obtain the following relation from (\ref{eq:Phi}), in which the derivatives of $f$ have to be evaluated in $u=\frac{\ell_1 + \ell_2}{2}$ and $v=\ell_1\,\ell_2$:
\begin{eqnarray*}
\delta \Phi (0,0) 
&=& 
- \left(\Delta_{\II}(\delta H)\right)_{(0,0)}\,f_u- \left(\Delta_{\II}(\delta K)\right)_{(0,0)}\,f_v
\\
&=&
-\left(\frac{2}{\ell_1}(6\,a+b)+\frac{2}{\ell_2}(b+6\,c) \right)f_u
-\left(24\,a\,\frac{\ell_2}{\ell_1}+8\,b+24\,c\,\frac{\ell_1}{\ell_2}\right) f_v\,.
\end{eqnarray*}
Similarly, from (\ref{eq:Psi}) we obtain the following equation:
\begin{eqnarray*}
\delta \Psi (0,0) 
&=&
\frac{1}{2} \left(\Delta (\delta H)\right)_{(0,0)} \,f_{uu}
+ \frac{1}{2} \left(\Delta (\delta K)\right)_{(0,0)}\,f_{uv}\\
&&
+\ell_1\,\ell_2\,\left(\Delta_{\II}(\delta H)\right)_{(0,0)} \,f_{uv}
+\ell_1\,\ell_2\,\left(\Delta_{\II}(\delta K)\right)_{(0,0)} \,f_{vv}
\\
&=&
\left(6\,a+2\,b+6\,c\right)f_{uu}
\\
&&
\quad
+
\left(24\,a\,\ell_2+ 4\,b\,(\ell_1+\ell_2) +24\,c\,\ell_1\right)f_{uv}
+ 
\left(24\,a\,(\ell_2)^2+8\,b\,\ell_1\,\ell_2+24\,c\,(\ell_1)^2\right)f_{vv}\,.
\end{eqnarray*}

Since the relation
$\delta \Phi (0,0)= C\, \delta \Psi (0,0)$ holds true for any $a,b,c$, we deduce the following three partial differential equations for the function $f$:
\begin{equation}
\label{eq:three_pde}
\left\{
\begin{array}{rcrcl}
\textrm{(i).}&\ \ &
\frac{-12}{\ell_1}\,f_u-24\,\frac{\ell_2}{\ell_1}\,f_v 
&=&
6\,C\,f_{uu} + 24\,C\,\ell_2\,f_{uv}+24\,C\,(\ell_2)^2\,f_{vv}\,;
\\
\textrm{(ii).}&\ \ &
-\left(\frac{2}{\ell_1}+\frac{2}{\ell_2}\right)\,f_u-8\,f_v 
&=&
2\,C\,f_{uu} + 4\,C\,(\ell_1 + \ell_2)\,f_{uv}+8\,C\,\ell_1\,\ell_2\,f_{vv}\,;
\rule{0pt}{15pt}
\\
\textrm{(iii).}&\ \ &
\frac{-12}{\ell_2}\,f_u-24\,\frac{\ell_1}{\ell_2}\,f_v 
&=&
6\,C\,f_{uu} + 24\,C\,\ell_1\,f_{uv}+24\,C\,(\ell_1)^2\,f_{vv}\,.
\rule{0pt}{15pt}
\end{array}
\right.
\end{equation}
Here we still assume the derivatives of the function $f$ to be evaluated in $u=\frac{\ell_1 + \ell_2}{2}$ and $v=\ell_1\,\ell_2$. The coefficients of 
(\ref{eq:three_pde}.ii) can all be expressed in terms of $u$ and $v$. For instance, there holds $\frac{1}{\ell_1}+\frac{1}{\ell_2}=\frac{2\,u}{v}$\,.
This can also be done for the other two equations, after they have been replaced by their sum and $\ell_1\times$(\ref{eq:three_pde}.i)+$\ell_2\times$(\ref{eq:three_pde}.iii), respectively. As such we can rewrite the system (\ref{eq:three_pde}) as
\begin{equation}
\label{eq:three_pde_2}
\left\{
\begin{array}{rcrcl}
\textrm{(i).}&\ \ &
-2\frac{u}{v}f_u + (4-8\frac{u^2}{v})f_v
&=&
C\,f_{uu} +4\,C\,u\,f_{uv} + 8\,C\,(u^2-\frac{v}{2})\,f_{vv}\,;
\\
\textrm{(ii).}&\ \ &
-\frac{2\,u}{v}\,f_u-4\,f_v 
&=&
\,C\,f_{uu} + 4\,C\,u\,f_{uv}+4\,C\,v\,f_{vv}\,;
\rule{0pt}{15pt}
\\
\textrm{(iii).}&\ \ &
-2\,f_u - 4\,u\,f_v
&=&
C\,u\,f_{uu} +4\,C\,v\,f_{uv} + 4\,C\,u\,v\,f_{vv}\,.
\rule{0pt}{15pt}
\\
\end{array}
\right.
\end{equation}
The difference between 
(\ref{eq:three_pde_2}.i) and (\ref{eq:three_pde_2}.ii) is a first-order differential equation in $f_v$ with respect to the variable $v$, which can easily be integrated with the following result:
\[
f(u,v) = Q(u)\,v^{(C-1)/C} + R(u)\,,
\] 
for some functions $Q$ and $R$.
If this expression is substituted in (\ref{eq:three_pde_2}.iii), there results
\begin{eqnarray*}
&-2\,Q'(u)\,v^{(C-1)/C} -2\,R'(u) - 4\frac{(C-1)}{C}\,u\,Q(u)\,v^{-1/C}
\qquad\qquad\qquad\qquad\qquad\qquad\qquad\qquad\qquad\qquad&
\\
&
=
C\,Q''(u)\,u\,v^{(C-1)/C}
+C\,u\,R''(u)+4\,(C-1)\,Q'(u)\,v^{(C-1)/C}
-4\,\frac{(C-1)}{C}\,u\,Q(u)\,v^{-1/C}\,.
&
\end{eqnarray*}
A similar equation results if we substitute our expression for $f$ in (\ref{eq:three_pde_2}.ii). A combination of this information easily results in $Q'(u)=R'(u)=0$. 
A comparison with (\ref{eq:sphere}) shows that the constant $R$ vanishes, which finally results in the expression
\[
f(u,v) = q_1\, v^\alpha  \qquad \textrm{(for $u>0$ and $v>0$).}
\]
for the function $f$,
where $\alpha=\frac{(C-1)}{C}$. We remark that the case $\alpha=1$ cannot occur and that $C=\frac{1}{(1-\alpha)}$. 

\textit{Case 2}. $H<0$ and $K>0$. This is essentially not different from the previous case: by reversing the unit normal vector field we switch from case 2 to case 1.

\textit{Case 3}. In the case where $K<0$, an analysis of the condition $\Phi = C\,\Psi$ for a similar family of surfaces, where  now $\ell_1 < 0 < \ell_2$, leads likewise to the conclusion $f(u,v) = q_3\, |v|^\alpha$.
This finishes the proof of part (ii).\ of the theorem.

Consider now a smooth function $f:\mathscr{D}\mapsto \mathbb{R}$ which is given by (\ref{eq:f}). From (\ref{eq:Phi}) and (\ref{eq:Psi}) we find that (for some $i$) 
\[
\Phi
=
-2\,q_i\,|K|^{\alpha}\,H-q_i\,\alpha\,|K|^{\alpha-1}\,\Delta_{\II}|K|
+q_i\left(-\alpha^2+\frac{3}{2}\,\alpha\right)\,|K|^{\alpha-2}\,\II(\grad_{\II}K,\grad_{\II}K)
=\frac{1}{(1-\alpha)}\Psi\,.
\]
This means that, for this choice of $f$, 
\begin{equation}
\label{eq:deltaA_deltaF}
\delta \textrm{Area}_{(N_f^M)} = \frac{1}{(1-\alpha)}\,\delta \Ff 
\end{equation}
for every deformation of every non-degenerate surface.
This finishes the proof of part (i).\ of the theorem.
\end{proof}

\begin{remark}
It can be asked whether there is also a simple relation between the 
\textit{second} variations of the functionals
$\textrm{Area}_{(N_f^M)}$ and $\Ff$ constructed from $f(u,v)=|v|^{\alpha}$. Restricting attention to locally strongly convex surfaces $M$ which are a critical point of the functional defined w.r.t. $f(u,v)=v^{\alpha}$, the second-order variation of  $\mathscr{F}_{f}$ is negative-definite for $\alpha=\frac{1}{4}$ or $\alpha\in\left[\,\frac{1}{2}\,,\,1\,\right[$, and positive-definite for 
$\alpha\in\left]\,1\,,\,+\infty\,\right[$.
This was shown by E. Calabi and  M. Wiehe (\cite{calabi1982} and \cite{wiehe1998}, theorem 4.3.). 

I am not aware of similar results for the second variation of $\textrm{Area}_{(N_f^M)}$ for $f(u,v)=v^{\alpha}$. 
\end{remark}

\begin{remark}
A result which is similar to the last theorem is described in \cite{binder_wiehe2005} (see also \cite{keilbach2000}, Ch. 3 and \cite{pabel_keilbach2001}). 
In contrast with the present article, in which Manhart's family is characterised among a collection of relative normal vector fields depending on a function $f$ of two variables, the authors of \cite{binder_wiehe2005} consider a collection of relative normal vector fields which depends on two real numbers. 
As follows from prop. 7.(i) of \cite{binder_wiehe2005}, 
Manhart's one-parameter family
can be characterised among this two-parameter collection of relative normal vector fields in a similar way as in the above theorem \ref{thm:manhart}.
\end{remark}

\begin{remark}
Let $\alpha\in\mathbb{R}\setminus\{0\}$ be a fixed number. 
For an ovaloid $M$, the gauge surface $\mathscr{M}$ which is
described by Manhart's relative normal vector field of $M$ (with parameter 
$\alpha$) can be considered. Then $M$ is up to translation determined by $\mathscr{M}$. Compare with Blaschke's comment on Minkowski's article in \cite{blaschkeII}, \S\,75.9.
\end{remark}

\section{A Characterisation of the Sphere Related with the Support Function.}

The theorem below, and particularly the integral formula (\ref{eq:int_for}), generalises results of K.-P. Grotemeyer ($\sigma=1$),  U. Simon ($\sigma=0$), and R. Schneider ($\sigma=-1$) (see \cite{grotemeyer1953}; \cite{schneider1966} and \cite{simon1967}, Satz 6.1). 

We will denote $P$ for the restriction of the position vector field along a surface, $P^{\tangentpart}$ for its tangent part, and $\rho=\llangle P,-N\rrangle$ for the support function.

\begin{theorem}
\label{thm:sup}
Assume $f:\mathbb{R}_0^+\rightarrow\mathbb{R}_0^+$ is a decreasing function and $-1\leqslant \sigma \leqslant 1$. 
Then every ovaloid in $\mathbb{E}^3$ which satisfies
\begin{equation}
\label{eq:rho}
\rho = f(H^{\sigma} \sqrt{K}^{1-\sigma})
\end{equation}
is a sphere.
\end{theorem}
\begin{proof} 
It is easy to see that $\textrm{div}(P^{\tangentpart})=2-2\,\rho H$. 
Consequently, for every $\phi\in \mathfrak{F}(M)$, there holds
\begin{equation}
\nonumber
\int \llangle P, \grad \phi\rrangle \dd\,\Omega 
=
\int \left\lgroup
\textrm{div}(\phi\,P^{\tangentpart}) - \phi\, \textrm{div}(P^{\tangentpart})
\right\rgroup  \dd\,\Omega
=
2\int (\rho H -1)\, \phi \,\dd\Omega\,.
\end{equation}
A consideration of the previous formula with $\phi=H^{\sigma} \sqrt{K}^{1-\sigma}$ yields the following result:
\begin{eqnarray}
\nonumber
&\displaystyle\frac{1}{2}\int \llangle P, \grad\left(H^{\sigma} \sqrt{K}^{1-\sigma}\right) \rrangle \,\dd\Omega
\qquad\qquad\qquad\qquad\qquad\qquad\qquad\qquad\qquad\qquad &\\
\label{eq:int_for}
&\displaystyle\qquad\qquad\qquad\qquad= \int \rho\,\sqrt{K}^{1-\sigma} \left(H^{1+\sigma}-\sqrt{K}^{1+\sigma}\right)\,\dd\Omega
+\int H^{\sigma}\left(H^{1-\sigma}-\sqrt{K}^{1-\sigma}\right)\,\dd\Omega
\,\geqslant\, 0\,,&
\end{eqnarray}
in which equality occurs precisely for the spheres.
It follows from (\ref{eq:rho}) that the left-hand side of the above equation is negative, and this concludes the proof.
\end{proof}

\section{A Variational Characterisation of the Sphere.}

\begin{theorem}
Let $\alpha\in\,\left]-\infty\,,\,0\,\right]$. The spheres are the only ovaloids in $\mathbb{E}^3$ which are a critical point of the functional $\int K^{\alpha}\,\dd\Omega$ under volume constraint.
\end{theorem}
\begin{proof}
Because the proof of this theorem is very similar to the proof of the last theorem of \cite{verpoort_ria}, we will not give much details.

First of all, because of theorem \ref{thm:manhart}, and more precisely equation (\ref{eq:deltaA_deltaF}), an ovaloid $M$ solves the variational problem which is stated in the theorem if and only if its relative mean curvature w.r.t. the relative normal vector field $N_f^M$ from 
(\ref{eq:relnorm}) with $f(u,v)=v^{\alpha}$ is constant. It can then be shown that $N_f^M$ is proportional to the position vector field $P$ for an appropriate choice of origin. This implies that $\rho=C\,K^{\alpha}$ for some constant $C$, and theorem \ref{thm:sup} finishes the proof.
\end{proof}

\subsection*{Acknowledgement.}
It is a pleasure to thank prof.\ dr.\ F.\ Manhart for his interest in this work and for useful discussions.

\end{document}